\documentclass[12pt]{amsart}

\numberwithin{equation}{section}

\usepackage{amsfonts,amssymb,amsmath,amsthm}
\usepackage{url}
\usepackage{enumerate}
\usepackage{color}
\usepackage[colorlinks]{hyperref}

\newcommand{\hilbp}{{\mathcal{H}\textnormal{ilb}_{p(t)}^n}}

\theoremstyle{plain}
\newtheorem{theorem}{Theorem}[section]
\newtheorem{corollary}[theorem]{Corollary}
\newtheorem{proposition}[theorem]{Proposition}
\newtheorem{lemma}[theorem]{Lemma}

\theoremstyle{definition}
\newtheorem{definition}[theorem]{Definition}

\theoremstyle{remark}
\newtheorem{remark}[theorem]{Remark}
\newtheorem{example}[theorem]{Example}

\begin{document}

\title{Initial ideal of a general rational or elliptic curve on quadrics}

\author[F.~Cioffi]{Francesca Cioffi}
\email{\href{mailto:cioffifr@unina.it}{cioffifr@unina.it}}

\author[D.~Franco]{Davide Franco}
\email{\href{mailto:davide.franco@unina.it}{davide.franco@unina.it}}

\author[G.~Ilardi]{Giovanna Ilardi}
\email{\href{mailto:giovanna.ilardi@unina.it}{giovanna.ilardi@unina.it}}

\address{Dipartimento di Matematica e Applicazioni \lq\lq R. Caccioppoli\rq\rq \\ Universit\`a degli Studi di Napoli Federico II\\ 
         Complesso Universitario di Monte S. Angelo\\ Via Cintia \\ 80126 Napoli \\ Italy.}
\thanks{The first and the third authors are members of GNSAGA (INdAM)}


\begin{abstract}
Over an algebraically closed field of characteristic zero, we prove that the generic initial ideal with respect to the degree reverse lexicographic term order of a general rational or elliptic curve on quadrics is almost revlex. 

Following constructive arguments, our proof combines features of a pencil of quadrics, interpolation methods and the notion of double generic initial ideal.
\end{abstract}

\maketitle

\section*{Introduction}

Over an algebraically closed field of characteristic zero, we investigate the following natural question: {\em is the generic initial ideal with respect to the degree reverse lexicographic term order of a general non-special curve an almost revlex ideal?}

One motivation for studying this question is its connection with the n-Strong Lefschetz Property (n-SLP, for short), due to the role played by almost revlex ideals in that context (see \cite[Corollary~13]{HW}, \cite[Chapter 6]{HMMNWW}, and \cite[Proposition 3.3, Theorems~5.4 and~5.6]{PaTo}; also see Corollary \ref{cor:SLP}).

Another reason, which is the focus of this work, is that the generic initial ideal of a general non-special curve has a special role in the irreducible component~$\mathcal Z_{p(t)}^n$ of the Hilbert scheme $\hilbp$ containing the general non-special curves with Hilbert polynomial~$p(t)$. It is called the {\em double-generic initial ideal} of $\mathcal Z_{p(t)}^n$ (see Section~\ref{sec:double}).

After showing that the answer to our question is positive for some general arithmetically Cohen–Macaulay non-special curves, we focus on general rational and elliptic curves. By geometrical arguments, we obtain a positive answer when such curves lie on quadrics (see Theorem~\ref{th:Interpolation}), namely:

\vskip 2mm
\noindent {\em The initial ideal of a general rational (respectively~elliptic) curve $C_d\subset\mathbb P^n_K$ of degree $d$ is almost revlex for every $d$ such that $n\leq d < \frac{n^2+3n}{4}$ (respectively~$n < d < \frac{n^2+3n+2}{4}$).}
\vskip 2mm

One of the first cases of curves not lying on a quadric is the general space rational curve $C_5$ of degree $5$, which has been studied and solved independently in \cite[Example~3.8]{FG} and in \cite[Section 5.2]{LRichard} by different methods. The approach of \cite{FG} examines all the strongly stable ideals with the same Hilbert function as $C_5$ and identifies the generic initial ideal by exclusion. 
Following the same line of the approach of \cite{LRichard}, thanks to a result of Ballico and Migliore (see~\cite{BM1990}) we face our problem performing a preliminary screening of the strongly stable ideals that could be the initial ideal of a general rational or elliptic curve $C_d$ of a given degree $d$ in $\mathbb P^n_K$, based on the general hyperplane section of such a curve (see Propositions~\ref{prop: sezione iperpiana} and~\ref{cor: componente Reeves}).

Next, focusing on the case of curves on quadrics, we study the pencil of quadrics containing a non-degenerate irreducible curve of degree $d-1$ and prove the existence of a line $L$ that intersects the curve transversely at a unique point (see Theorem~\ref{th:generic line} and Proposition \ref{prop: transversally}).

By appropriately choosing two quadrics and applying an interpolation method based on the Möller–Buchberger approach, we show that starting from a general rational (respectively, elliptic) curve $C_{d-1}$, the curve $C_{d-1} \cup L$ has an initial ideal with the expected behavior, at least in the initial degree. Since the curve $C_{d-1} \cup L$ is smoothable in the Hilbert scheme of general rational (respectively, elliptic) curves of degree $d$, we conclude our result using both the preliminary screening of potential initial ideals and the notion and properties of the double-generic initial ideal.

\section{Almost revlex ideal and general non-special curves}\label{sec:almost}

Let $R:=K[x_0,\dots,x_n]$ be the standard graded polynomial ring over an algebraically closed field $K$ of null characteristic in $n+1$ variables, provided with the degree reverse lexicographic  (degrevlex, for short) term order $\prec$, with $x_0\prec \dots \prec x_n$. 

For a homogeneous ideal $I$ of $R$, $I_t$ is the $K$-vector space generated by the homogeneous polynomials of $I$ of degree $t$ and $I_{\geq t}$ is the ideal generated by the homogeneous polynomials of $I$ of degree $\geq t$ and $I^{sat}$ is the {\em saturation} of $I$. 
The {\em satiety} $\mathrm{sat}(I)$ is the minimal non-negative integer $m$ such that $I_m={I^{sat}}_m$.  
The {\em regularity} $\mathrm{reg}(I)$ is the minimum integer $m$ such that the $i$-th syzygies of $I$ are generated in degree $\leq m+i$, for every $i\geq 0$.

A term of $R$ is a power product $x^\alpha=x_0^{\alpha_0}\dots x_d^{\alpha_d}$ with $(\alpha_0,\dots,\alpha_d)\in \mathbb Z_{\geq 0}^{d+1}$. If $x^\alpha$ is different from $1$, then we denote by $\min(x^\alpha)$ the minimal variable that divides $x^\alpha$. We will denote by $\mathbb T$ the set of all the terms of $R$. 

If $J$ is a monomial ideal, we denote by $B_J$ the minimal set of terms generating $J$ and by $\mathcal N(J)$ the {\em sous-escalier} of $J$, that is the set of terms of $R$ outside $J$. 

\begin{remark}\label{rem:drl} 
Let $N_t$ be a set of terms of the same degree $t>0$ and, for every $i=0,\dots,n$, assume that the set $N_{t,i}:=\{\tau \in N_t : x_i\leq \min(\tau)\}$ is ordered in increasing order with respect to $\prec$. Then,
$$x_0N_{t,0} \sqcup \dots \sqcup \ x_iN_{t,i}\sqcup \dots \sqcup x_nN_{t,n}$$ 
is the list of elements of $\{x_i\tau \ \vert \ \tau\in N_t, i\in \{0,\dots,n\}\}$ without repetitions in increasing order with respect to $\prec$. It is called the {\em expansion} of $N_t$ at degree $t+1$.
\end{remark}

A monomial ideal $J$ is {\em strongly stable} if $x^\alpha\in J$ implies $x^\alpha x_i/x_j \in J$ for every $x_j$ by which $x_\alpha$ is divisible and $x_j\prec x_i$.

Let $I$ be an ideal of $R$ and $f$ a non-null polynomial. With the usual language in the theory of Gr\"obner bases, we denote by $\mathrm{in}(f)$ the initial term of $f$ and by $\mathrm{in}(I)$ the initial ideal of $I$ with respect to $\prec$. Recall that $\mathrm{in}(f)$ is divisible by $x_0^k$ if and only if $f$ is divisible by $x_0^k$, for every $k\geq 1$.

It is well-known that there is an open subset $U$ of $GL_d(K)$ such that, for every $g\in U$, $\mathrm{in}(g(I))$ is a constant Borel ideal, which is called the {\em generic initial ideal} of $I$ and denoted by $\mathrm{gin}(I)$ \cite{GA,BS}. In characteristic zero a Borel ideal is strongly stable (e.g.~\cite[Proposition~1.13]{D}), and the converse is true in any characteristic. 

\begin{lemma}\label{lemma:BS2} \cite{BS2}
Let $I\subset R$ be an homogeneous ideal and $\prec$ be the degrevlex term order. 
\begin{itemize}
\item[(i)] $(\mathrm{in}(I),x_0)=\mathrm{in}(I,x_0)$; 
\item[(ii)] if $I$ is saturated, then $\mathrm{gin}(I)$ and $\mathrm{in}(I)$ are saturated;
\item[(iii)] $\mathrm{reg}(\mathrm{gin}(I)) = \mathrm{reg}(I) \leq \mathrm{reg}(\mathrm{in}(I))$.
\end{itemize}
\end{lemma}

\subsection{Almost revlex ideal}

\begin{definition}\label{def:segment}
Given a non-negative integer $t$, a set $L$ of terms of degree $t$ is a {\em $\prec$-segment} if for every couple of terms $\tau\in L$ and $\tau'$ of degree $t$, $\tau\prec \tau'$ implies that $\tau'$ belongs to~$L$. A monomial ideal $J\subset R$ is a {\em $\prec$-segment ideal} if the set of terms in $J$ of degree $t$ is a $\prec$-segment, for every degree $t$. 
\end{definition}

A $\prec$-segment ideal is strongly stable by construction.

Recall that there is an non-empty Zariski open subset of $s$ points with maximal Hilbert function, for every integer $s$.
So, meaning that {\em general} points are a reduced projective zero-dimensional scheme $Z$ with maximal Hilbert function, the generic initial ideal of the defining ideal of general points is a $\prec$-segment ideal (see \cite[Theorem~4.4]{MR99} and \cite[Theorem~1.2]{CoSi}), and this is the only {\em saturated} $\prec$-segment ideal with Hilbert polynomial $d$ (see~\cite[Remark~3.14]{CLMR}). 

We set $\mathrm{in}(Z):=\mathrm{in}(I(Z))$ and $\mathrm{gin}(Z):=\mathrm{gin}(I(Z))$. Thanks to the generality of the points, $\mathrm{in}(Z)$ and $\mathrm{gin}(Z)$ coincide.

\begin{definition} \label{def:revlex}{\rm (see \cite{D})} 
A monomial ideal $J\subset R$ is an {\em almost reverse lexicographic ideal} (or {\em weakly reverse lexicographic ideal} or {\em almost revlex ideal}, for short) if, for every term $\tau\in B_J$ and term $\tau'$ of the same degree of $\tau$, $\tau \prec \tau'$ implies that $\tau'$ belongs to $J$.
\end{definition}

An almost revlex ideal is strongly stable by construction. 

In \cite{Pardue}, K. Pardue characterizes all the Hilbert functions $H$ for which there exists an almost revlex ideal $J\subseteq R$ such that $R/J$ has Hilbert function $H$. In this case we say that $H$ {\em admits} the almost revlex ideal. 

If $H$ is a Hilbert function that admits both the $\prec$-segment ideal and the almost revlex ideal, then these ideals coincide. Indeed, a $\prec$-segment ideal is also an almost revlex ideal, but in general an almost revlex ideal is not a $\prec$-segment ideal.

\subsection{Hilbert function of a general non-special curve}
We will freely use the common notation of sheaf cohomology in $\mathbb P^n_K$.



A curve $C\subset \mathbb P^n_K$ is a closed subscheme of $\mathbb P^n_K$ of dimension $1$. 
We denote by $p_C(t)$ the Hilbert polynomial of $C$, by $H_C(t)$ its Hilbert function and by $\rho_C:=\min\{t\geq 0 : H_C(v)=p_C(v), \forall \ v\geq t\}$ the regularity of $H_C$.

Recall that the Castelnuovo-Mumford regularity $\mathrm{reg}(C)$ of $C$ is equal to the regularity $\mathrm{reg}(I(C))$ of the (homogeneous saturated) defining ideal $I(C)$ of $C$ and $\mathrm{reg}(C)\geq \rho_C+1$. 

The {\em deficiency module (or Rao module)} $M(C)$ of $C$ is the direct sum $\bigoplus H^1(\mathcal I_C(t))$.

We denote by $\hilbp$ the Hilbert scheme parametrising all subschemes with Hilbert polynomial $p(t)$ in the projective space $\mathbb P^n_K$. A point of $\hilbp$ is identified with the saturation of every homogeneous ideal defining it. 

For every $n\geq 3$ and non-negative integer $d\geq n+g$, we consider the Hilbert polynomial $p(t)=dt+1-g$ and the open irreducible subset $U_{p(t)}^n$ of $\hilbp$ parametrising the smooth, irreducible, non-degenerate curves (see \cite{BERavello}, \cite[page 59-61]{Harris1982}, \cite{BaFo}). 
The closure $\mathcal Z_{p(t)}^n$ of $U_{p(t)}^n$ in the Hilbert scheme is an irreducible component of $\hilbp$ of dimension $4g-3 +(r+1)(d-g+1)-1$.

A statement about a {\em general} curve will mean that the corresponding statement holds for all the curves in an irreducible  open dense subset of $\mathcal Z_{p(t)}^n$.

\begin{definition}\label{def: maximal rank}
Let $C$ be a curve in $U_{p(t)}^n \subseteq\hilbp$ and consider the natural map of restriction 
\begin{equation}\label{map}
\varrho(t): H^0(\mathbb P^n_K, \mathcal O_{\mathbb P^n_K}(t)) \rightarrow H^0(C,\mathcal O_C(t)).
\end{equation}  
The maximal rank conjecture says: {\em for all $t$, the map $\varrho(t)$ has maximal rank, i.e.~is either injective or surjective} (see \cite[page 79]{Harris1982}, \cite[page 6]{EisHarris83}). The curve {\em $C$ has maximal rank} if this property is satisfied by $C$.
\end{definition}

\begin{definition}\label{def: non-special}
A curve $C$ is said {\em non-special} if its hyperplane section is non-special, i.e.~$H^1(C,\mathcal O_C(1))=0$, and hence $H^1(C,\mathcal O_C(t))=0$ for every $t\geq 1$. 
The {\em index of speciality} of $C$ is $e(C):=\max\{\ell\in \mathbb Z : h^2(\mathcal I_C(\ell))\not=0\}= \max\{\ell \in \mathbb Z : h^1(\mathcal O_C(\ell))\not=0\}$.
\end{definition}

It is noteworthy that {\em smooth rational and smooth elliptic curves are always non-special}. 

A smooth irreducible non-special curve $C$ has maximal rank if and only if $H_C(t)=\min\{\binom{n+t}{n}, p_C(t)\}$, that is $C$ has maximal Hilbert function (e.g. \cite[Remarks 3.3(2)]{BE1983}). 
Like the following example shows, there are non-special curves without maximal rank.

\begin{example}
Let $C$ be the rational curve given by the rational map $\Phi: \mathbb P^1_K \rightarrow \mathbb P^3_K$ such that $\Phi (u,v):=(u^9+v^9, u^8v+uv^8,u^7v^2+v^9, u^2v^7)$. The Hilbert function of $C$ is $(1, 4, 10, 20, 33, 44, 54, 9t+1)$. This rational curve is smooth and hence non-special, but it has not maximal rank.
\end{example}

By several papers (see 
\cite{BE1987} and the references therein), E.~Ballico and Ph.~Ellia proved the maximal rank conjecture for {\em general} non-special curves in $\mathbb P^n_K$. 
This means that, for every $n\geq 3$, there is a dense open subset $\mathcal M^n_{p(t)}$ in $\mathcal Z_{p(t)}^n$ parametrizing  smooth irreducible non-special curves $C\subset \mathbb P^n_K$ with arithmetic genus $g\geq 0$, degree $d\geq n+g$ and maximal rank. 

In the following, {\em (generic) initial ideal} of a curve $C$  means the {\em (generic) initial ideal} of $I(C)$ with respect to the degrevlex term order~$\prec$. If $C$ is a general curve, then its generic initial ideal and its initial ideal coincide. 

\begin{proposition}\label{prop: esistenza almost}
The Hilbert function $H(t)= \max\{dt+1-g, \binom{n+t}{t}\}$ of a general non-special curve $C_d$ of degree $d$ in $\mathbb P^n_K$ admits a saturated almost revlex ideal. 
\end{proposition}

\begin{proof}  
Thanks to \cite[Theorem 4]{Pardue} every  Hilbert function $H(t)= \max\{dt+1-g, \binom{n+t}{t}\}$ of a general non-special curve $C_d$ of degree $d$ in $\mathbb P^n_K$ admits an almost revlex ideal $J$. In order to prove that $J$ is saturated it is enough to observe that every term in $B_J$ is not divisible by the last variable $x_0$. This can be done by a standard combinatorial computation based on Remark \ref{rem:drl} and the expansion of the sous-escalier~$\mathcal N(J)$ of $J$. 
\end{proof}

\begin{corollary}\label{cor:grado iniziale}
Let $\alpha_{d-1}=\min\{t\in \mathbb N : \binom{n+t}{n}>(d-1)t+1-g\}$ be the initial degree of a general non-special curve $C_{d-1}$ of degree $d-1$ and $\alpha_{d}=\min\{t\in \mathbb N : \binom{n+t}{n}>dt+1-g\}$ be the initial degree of a general non-special curve $C_{d}$ of degree $d$, then $\alpha_{d-1} \leq \alpha_d \leq \alpha_{d-1}+1$.
\end{corollary}

\begin{proof}
The first inequality is immediate. The second inequality follows by a standard combinatorial computation based on Remark \ref{rem:drl} and the expansion of the sous-escalier of the initial ideal of $C_d$, which is saturated by Lemma \ref{lemma:BS2}. 
\end{proof}

\begin{example}
When the curve is not {\em enough general}, the maximal rank property does not guarantee that the initial ideal of a rational curve is almost revlex. For example, let $C$ be the rational normal curve given by the rational map $\Phi: \mathbb P^1_K \rightarrow \mathbb P^4_K$ such that $\Phi (u,v):=(u^4,v^4,u^3v, u^2v^2, uv^3)$. Hence, $C$ is non-special and has maximal rank. However, its initial ideal $(x_2^2, x_2x_3, x_3^2, x_0x_1, x_1x_2, x_1x_3)$ is not an almost revlex ideal. Observe, that $x_0$ is a zero-divisor on $R/I(C)$ although $I(C)$ is saturated. This means that $C$ is not in {\em generic} coordinates.
\end{example}

\subsection{General non-special arithmetically Cohen-Macaulay curves} 

\begin{lemma}\label{lemma:aCM}
Let $C_d\subset\mathbb P^n_K$ be a general non-special curve of arithmetic genus $g$ and degree $d\geq n+g$. Then, $C_d \text{ is arithmetically Cohen-Macaualy} \Longleftrightarrow \rho_{C_d} \leq 1 \Longleftrightarrow d=n+g \text{ and } g\leq \binom{n}{2}$.
\end{lemma}

\begin{proof} 
First recall that ${C_d}$ is arithmetically Cohen-Macaulay if and only if its Hartshorne-Rao module $\bigoplus H^1(\mathbb P^n_K,\mathcal I_{C_d}(t))$ is null. 
Then, we have $H_{C_d}(t)=H_{C_d}(t)+h^1(\mathbb P^n_K,\mathcal I_{C_d}(t))=h^0(C_d,\mathcal O_{C_d}(t))=P_{C_d}(t)=dt+1-g$, for every $t>0$, because ${C_d}$ is non-special. Further, $H^1(\mathbb P^n_K,\mathcal I_{C_d}(t))=0$ for every $t\leq 0$ because ${C_d}$ is irreducible.

About the second equivalence, if $\rho_{C_d}\leq 1$ then $n+1=H_{C_d}(1)=P_{C_d}(1)=d+1-g$, hence $d=n+g$, and recalling that ${C_d}$ has maximal rank we have also $H_{C_d}(2)=P_{C_d}(2)=2(n+g)+1-g \leq \binom{n+2}{2}$, so $g\leq \frac{n^2-n}{2}$ and vice versa.
\end{proof}

Thanks to Lemma \ref{lemma:BS2}, if $C\subset \mathbb P^n_K$ is arithmetically Cohen-Macaulay with general hyperplane section $Z$, then $\mathrm{gin}(C)$ is completely determined by $\mathrm{gin}(Z)$, and the following result holds.

\begin{proposition}\label{prop: minimal degree}
Let $C\subset \mathbb P^n_K$ be a general non-special curve of arithmetic genus $g\geq 0$ and minimal possible degree $d=n+g$. If $g\in\{0,1,2,\binom{n}{2}-2,\binom{n}{2}-1, \binom{n}{2}\}$ then $\mathrm{in}(C)$ is almost revlex.
\end{proposition}

\begin{proof}
In any of the cases we are considering, the curve $C$ is aCM by Lemma \ref{lemma:aCM}. Thus, the variables $x_{0},x_1$ form a regular sequence for $R/I(C)$ and for $R/\mathrm{in}(C)$. So, $\mathrm{in}(C)$ is generated by terms not divisible by $x_0, x_1$. Moreover, $\mathrm{in}(C)$ has the same generators as $\mathrm{in}(Z)$, where $Z$ is the hyperplane section of $C$ by $H:x_0=0$. 

If $g=\binom{n}{2}$ then $\mathrm{in}(C)=(x_0,\dots,x_{n-2})^3$ by Hilbert function arguments, and we conclude.

In the remaining cases, we observe that $\mathrm{in}(Z)=\frac{(\mathrm{in}(I(C)),x_0)}{(x_0)}$ and then $H_Z(t)=\Delta H_C(t)=(1,n,n+g,\dots)$. By Hilbert function arguments and properties of strongly stable ideals, $\mathrm{in}(Z)$ is the ideal generated by 
\vskip 1mm
$\{\tau \in \mathbb T : \tau \in (x_0,\dots,x_{n-2})_2\}, \text{ if } g=0$,

$\{x_{n-2}^3\}\cup(\{\tau \in \mathbb T : \tau \in (x_0,\dots,x_{n-2})_2\}\setminus\{x_{n-2}^2\}), \text{ if } g=1$,

$\{x_{n-2}^3,x_{n-2}^2x_{n-3} \}\cup (\{\tau \in \mathbb T : \tau \in (x_0,\dots,x_{n-2})_2\}\setminus\{x_{n-2}^2,x_{n-2}x_{n-3}\}), \text{ if } g=2$,

$\{x_0^2\}\cup\{\tau \in \mathbb T : \tau \in (x_0,\dots,x_{n-2})_3\}, \text{ if } g=\binom{n}{2}-1$,

$\{x_0^2, x_0x_1\}\cup\{\tau \in \mathbb T : \tau \in (x_0,\dots,x_{n-2})_3\}, \text{ if } g=\binom{n}{2}-2$.
\end{proof}

\begin{remark}\label{rem: minimal degree}
For $n=3$, the statement of Proposition \ref{prop: minimal degree} holds for every genus $g\in\{0,1,\dots,\binom{n}{2}\}$. Indeed, thanks to Lemma \ref{lemma:aCM}, it is sufficient to observe that, being $Z$ a subscheme of $\mathbb P^2_K$, there exists a unique saturated strongly stable ideal with the same Hilbert function as $Z$ and it is the revlex ideal necessarily, because $H_Z$ is maximal by construction. Hence, $\mathrm{gin}(Z)$ is the revlex ideal with maximal Hilbert function with degree $d$ and $\mathrm{gin}(C)$ is the almost revlex ideal with the same Hilbert function as $C$, because it has the same generators as $\mathrm{gin}(Z)$.
\end{remark}
 
\subsection{General non-special non-arithmetically Cohen-Macaulay curves} 
 
For general non-special non-arithmetically Cohen-Macaulay (non-aCM, for short) curves the following information will be useful.

\begin{proposition}\label{prop: grado alpha}
Let $C_d\subset \mathbb P^n_K$ be a general non-special non-aCM curve of degree $d$, $Z$ denote its general hyperplane section, and $\alpha_d$ be the initial degree of the defining ideal $I(C_d)$. Then, $\mathrm{gin}(C_d)$ is completely determined by the knowledge of $\mathrm{gin}(Z)$ and $\mathrm{gin}(C_d)_{\alpha_d}$.
\end{proposition}

\begin{proof}
Thanks to Lemma \ref{lemma:BS2},
$$\mathrm{gin}(Z)=((\mathrm{gin}(C_d),x_0)/(x_0))^{sat}.$$
Consequently, the equality $\mathrm{gin}(Z)_t=((\mathrm{gin}(C_d),x_0)/(x_0))_t$ holds for every degree $t$ such that ${((\mathrm{gin}(C_d),x_0)/(x_0))^{sat}}_t=((\mathrm{gin}(C_d),x_0)/(x_0))_t$. 

The equality $\mathrm{gin}(Z)_t=((\mathrm{gin}(C_d),x_0)/(x_0))_t$ holds for every $t\geq \alpha_d+1$, because $\alpha_d \leq \mathrm{reg}(C_d)\leq \alpha_d+1$.  
Hence, if we know $\mathrm{gin}(Z)$, then it is enough to study $\mathrm{gin}(C_d)$ at the degree $\alpha_d$ in order to completely know $\mathrm{gin}(C_d)$. 
\end{proof}

\section{The hyperplane section and the deficiency module of a general rational or elliptic curve}
\label{sec: hyperplane section}

From now a {\em general} non-special curve $C_d$ of degree $d$ means a curve parametrised by a point of the dense open subset $\mathcal M^n_{p(t)}$, where $p(t)=dt+1-g$ and $g$ is the genus of $C_d$.

In this section we focus on the general hyperplane section of general non special curves of genus $g=0,1$, that are general rational or elliptic curves. 

\begin{proposition}\label{lemma:lemma}
For every general property $\mathcal P$ of reduced zero-dimensional schemes of degree $d$ in $\mathbb P^{n-1}$, there exists a general rational or elliptic curve with general hyperplane section satisfying the property $\mathcal P$.
\end{proposition}

\begin{proof}
Let $V_d$ be the (irreducible) non-empty open subset parametrising the reduced schemes of $d$ points in the Hilbert scheme over $\mathbb P^{n-1}_K$ with Hilbert polynomial $d$ and let $\mathcal H$ be an hyperplane of $\mathbb P^n_K$. Let $U_d$ be the (irreducible) open subset of $\hilbp$ that parametrises the smooth curves of degree $d$ that intersect $\mathcal H$ transversally. Hence, $V_d\times U_d$ is irreducible too. Moreover, consider the following set
$$\Delta:=\{(C\cap \mathcal H, C) \ \vert \ C\in U_d\} \subseteq V_d\times U_d$$
and the two projections
$$p_1: \Delta \to V_d, \qquad p_2: \Delta \to U_d.$$
The projection $p_1$ is surjective, thanks to results of Ballico and Migliore described in \cite{BM1990} (see \cite[Theorem~1.6]{BM1990} for rational curves and \cite[Proposition~2.4]{BM1990} for elliptic curves, taking into account that general points are in general linear position). The projection $p_2$ is an isomorphism by construction. Hence, $\Delta$ is irreducible too.

Let now $V_d'$ be a non-empty open subset of $V_d$ corresponding to the general property~$\mathcal P$. Then $p_2(p_1^{-1}(V'_d))\cap \mathcal M^n_{p(t)}$ is a non-empty open set of curves with general hyperplane section satisfying the property $\mathcal P$.
\end{proof}

\begin{proposition}\label{prop: sezione iperpiana}
The general hyperplane section of a general rational or elliptic curve $C$ of degree $d$ in $\mathbb P^n_K$ is a zero-dimensional scheme $Z\subset \mathbb P^{n-1}_K$ such that $\mathrm{gin}(Z)$ is the only saturated~$\prec$-segment ideal with Hilbert polynomial $d$. 
\end{proposition}

\begin{proof}
It is enough to recall that the property to be a $\prec$-segment ideal is open, hence general (see the proof of \cite[Theorem~1.2]{CoSi} for the property to be a $\prec$-segment ideals for any term order~$\prec$). Then we conclude applying Proposition \ref{lemma:lemma}.
\end{proof} 

\begin{corollary}\label{cor: cor3}
Given a general rational or elliptic curve $C_d$, $\mathrm{gin}(C_d)$ coincides with the almost revlex ideal with the same Hilbert function as $C_d$ if and only if  
$\mathrm{gin}(C_d)_{\alpha_d}\cap \mathbb T$ is a $\prec$-segment.
\end{corollary}

\begin{proof}
The result follows straightforwardly from Propositions \ref{prop: grado alpha} and \ref{prop: sezione iperpiana}, and from Definition~\ref{def:revlex}.
\end{proof}

\begin{corollary}\label{cor: cor4}
Let $C_d$ be a general rational or elliptic curve of degree $d$ and $\alpha_d$ the initial degree of $I(C_d)$. If $\binom{n+\alpha_d}{\alpha_d}-p_C(\alpha_d)\leq 2$ or $p_C(\alpha_d)-\binom{n+\alpha_d-1}{\alpha_d-1}\leq 2$, then $\mathrm{gin}(C_d)$ is an almost revlex ideal.
\end{corollary}

\begin{proof}
It is sufficient to observe that $\dim(I(C_d))_{\alpha_d}=\binom{n+\alpha_d}{\alpha_d}-p_C(\alpha_d)$ and that $p_C(\alpha_d)\geq\binom{n+\alpha_d-1}{\alpha_d-1}$. Then we conclude thanks to the definition of strongly stable ideals and, if $p_C(\alpha_d)=\binom{n+\alpha_d-1}{\alpha_d-1}$, due to Corollary \ref{cor: cor3}, because $\mathrm{reg}(C_d)=\alpha_d$ in this particular case.
\end{proof}

If a curve $C$ is integral, then its Rao module $M(C)=\bigoplus H^1(\mathcal I_C(t))$ is a finitely generated $K$-vector space and hence also an (Artinian) finitely generated graded $R$-module (for example, see \cite[Proposition~2.1]{Kreuzer}). In the following statement a graded $R$-module $M$ is said to have {\em the maximal rank property} if there is a linear form $L\in R_1$ such that, for every integer $i\geq 0$, the multiplication map by $L$ from $M_i$ to $M_{i+1}$ has maximal rank, i.e.~is injective or surjective.

\begin{corollary}\label{cor: Rao module}
The deficiency module of a general rational or elliptic curve has the maximal rank property.
\end{corollary}

\begin{proof}
It is enough to observe that, by Proposition \ref{prop: sezione iperpiana}, a general rational or elliptic curve $C$ is {\em ordinary}, according to  \cite[Definition~2.3 and Remark~2.4]{MiroRoig}. Then the statement follows from \cite[Theorem~3.4 and Remark~3.9]{MiroRoig}.
\end{proof}

\section{Almost revlex ideal and double-generic initial ideal} \label{sec:double}

We recall that every closed irreducible subset $Y$ of a Hilbert scheme $\hilbp$ that is stable under the action of $GL_n(K)$ contains at least one point corresponding to a strongly stable ideal. Among all the points of $Y$ defined by strongly stable ideals, we can find a special strongly stable ideal $\mathbf{G}_Y$ which is the saturation of the generic initial ideal of the generic point of $Y$. This special strongly stable ideal $\mathbf{G} _Y$ is called the {\em double-generic initial ideal of $Y$}. This notion has been introduced and investigated for the first time in~ \cite{BCR2017} also in the more general setting of Grassmannian, with the terminology of extensors. 

Next result contains some of the main properties of a double-generic initial ideal. First, we need to recall the following definition. 

\begin{definition}\label{def:ordinamento}\cite{BCR2017}
Let $J$ and $H$ be two monomial ideals in $S$ such that $S/J$ and $S/H$ have the same Hilbert polynomial $p(t)$. Let $r$ be the Gotzmann number of $p(t)$ and consider the set of generators $B_{J_{\geq r}}=\{\tau_1,\dots,\tau_q\}$ and $B_{I_{\geq r}}=\{\sigma_1,\dots,\sigma_q\}$ ordered by the degrevlex term order, where $q=\binom{n+r}{n}-p(r)$. We write $J >\!\!> H$ if $\tau_i \geq \sigma_i$ for every $i\in\{1,\dots,q\}$.
\end{definition}

\begin{lemma}\label{lemma:dgii}\cite[Propositions 2 and 3; Definition 5; Theorems 3 and 4]{BCR2017} Let $Y$ be an irreducible closed subset of $\hilbp$ stable under the action of $GL_n(K)$ and $\mathbf{G}_Y$ its double-generic initial ideal.
\begin{itemize}
\item[(i)] For every ideal $I$ defining a point of $Y$, $\mathrm{gin}(I)$ and $\mathrm{in}(I)$ define points of $Y$.
\item[(ii)] There exists the maximum among all the Borel ideals defining points of $Y$ with respect to the partial order $>\!\!>$, and this maximum is $\mathbf{G}_Y$.
\item[(iii)] There is a non-empty open subset $V$ of $Y$ such that $\mathrm{gin}(I)=\mathrm{in}(I)=\mathbf{G}_Y$ for every saturated ideal $I$ defining a point in $V$.
\item[(iv)] The Hilbert function of $\mathrm{Proj}(R/\mathbf{G}_Y)$ is the maximum among the Hilbert functions of $\mathrm{Proj}(R/I)$, as $I$ varies among the ideals defining points of $Y$. 
\end{itemize}
\end{lemma}

Every irreducible component of $\hilbp$ is a closed irreducible subset stable under the action of $GL_n(K)$. The closed irreducible subset ${Z}_{p(t)}^n$ is stable under the action of $GL_n(K)$, because is the closure of the set of smooth irreducible non-special curves which is stable under the action of $GL_n(K)$. 

\begin{proposition}\label{prop:massimo}
The almost revlex ideal $J$ with the same Hilbert function as a general non-special curve $C\subseteq \mathbb P^n_K$ with genus $g\geq 0$ and degree $d\geq n+g$ is the double-generic initial ideal of every irreducible closed subset of $\hilbp$ stable under the action of $GL_n(K)$ and containing the point defined by $J$. 
\end{proposition}

\begin{proof} 
The ideal $J$ is the (saturated) almost revlex ideal corresponding to the maximal possible Hilbert function of a projective scheme with Hilbert polynomial $p(t)=dt+1-g$. Moreover, by construction $J$ is the maximum with respect to the partial order $>\!\!>$ among all the saturated strongly stable ideals with the same Hilbert function. Thus, we conclude by Lemma \ref{lemma:dgii}(ii) and (iv). 
\end{proof}

\begin{corollary}
The generic initial ideal of a general non-special curve is an almost revlex ideal $J$ if and only if $J$ is the double-generic initial ideal of $\mathcal Z^n_{p(t)}$.
\end{corollary}

\begin{proof}
It is sufficient to recall that $\mathcal M^n_{p(t)}$ is a dense open subset of $\mathcal Z^n_{p(t)}$ and then use Lemma \ref{lemma:dgii}(iii) observing that $\mathcal M^n_{p(t)} \cap V$ is non-empty.
\end{proof}

\begin{corollary}\label{cor: componente Reeves}
Given a general rational or elliptic curve $C$ of degree $d$, let $Y$ be the irreducible component of $\hilbp$ containing every strongly stable ideal $J$ such that $\Bigl(\frac{(J,x_0)}{(x_0)}\Bigr)^{sat}$ is the $\prec$-segment with the maximal Hilbert function of Hilbert polynomial $d$. 
\begin{itemize}
\item[(i)] The almost revlex ideal with the same Hilbert function as $C$ belongs to $Y$.
\item[(i)] $Y\cap \mathcal Z_{p(t)}^n$ is non-empty.
\end{itemize}
\end{corollary}

\begin{proof}
The existence of $Y$ is guaranteed by \cite[Theorem 6]{R1}. Then, the two items follow from Propositions \ref{prop: grado alpha} and \ref{prop: sezione iperpiana} and by the fact that the saturation of a strongly stable ideal is generated by the terms generating the ideal with the least variable replaced by~$1$.
\end{proof}

\begin{example}\label{ex: n=5 d=8}
In $\mathbb P^5_K$ let us consider the general rational curve $C_8$ of degree $d=8$. The initial ideal of its general hyperplane section $Z$ is the degrevlex-segment ideal 
$$\mathrm{in}(Z)=(x_2^3,x_2^2x_3,x_2^2x_4,x_2x_5,x_3^2,x_3x_4,x_3x_5,x_4^2,
x_4x_5,x_5^2)\subseteq K[x_1,\dots,x_5].$$
Beyond the almost revlex ideal with the same Hilbert function of $C_8$ another possible candidate for the initial ideal of $C_8$ is
$$J=(x_2x_5,x_3x_5,x_4x_5,x_5^2,x_1x_3^2,x_1x_3x_4,x_1x_4^2,
x_2x_3^2,x_2x_3x_4,x_2x_4^2,$$
$$x_3^3,x_3^2x_4,x_3x_4^2,x_4^3,x_2^3,x_2^2x_3,x_2^2x_4)\subseteq K[x_0,x_1,\dots,x_5]$$
because $\Bigl(\frac{(J,x_0)}{(x_0)}\Bigr)^{sat}=\mathrm{in}(Z)$.
\end{example}

\section{Initial ideal of a general rational or elliptic curve on quadrics}

In this section we consider general rational curves and general elliptic curves $C_d$ of degree $d$ such that $\alpha_d=2$. 



\begin{theorem}\label{th:generic line} {\rm [Generic Line]}
Let $C\subset \mathbb P^n_K$ be a non-degenerate irreducible curve  on at least two quadrics $g_1$ and $g_2$. 
There exists a point $P$ of $C$ in which a quadric $\Sigma$ of the pencil $\langle g_1, g_2\rangle$ is smooth. Then the generic line through $P$ is not contained in any quadric of the pencil.
\end{theorem}

\begin{proof}
First we observe that the point $P$ exists. 
Every quadric in the pencil $\langle g_1, g_2\rangle \subseteq H^0(\mathcal I_{C}(2))$ is reduced and irreducible because the polynomial which defines it can be considered as a minimal generator of the ideal of the irreducible curve $C$. 
Moreover, its singular locus is contained in hyperplanes and hence does not contain $C$ because $C$ is not degenerate. In conclusion, for every quadric $\Sigma$ in the pencil there is at least a point $P\in C$ in which $\Sigma$ is smooth.

The lines through $P$ in $\mathbb P^n_K$ form a projective space of dimension $n-1$, by the projection $\pi_P$ from $P$ to a general hyperplane $H\simeq \mathbb P^{n-1}_K$. Let $\tilde \Sigma$ be the hyperplane section of $\Sigma$ by~$H$. Then $\tilde \Sigma$ is irreducible, non degenerate and of dimension $n-2$ (for example, see \cite[Proposition~18.10]{Harris1995}).

We now claim that the lines through $P$ contained in $\Sigma$ form a closed subset $X_{\Sigma}$ of dimension $\leq n-3$ in $H$. Indeed, if this dimension were equal to $n-2$, then $X$ would be a component of $\tilde \Sigma$, which is irreducible. Hence $\tilde \Sigma=X_{\Sigma}$ and $\Sigma$ would be a cone with vertex in $P$. But $\Sigma$ is smooth in $P$, and so this is not possible.

As $\Sigma$ varies among all the quadrics that are smooth in $P$, the closed set $X_{\Sigma}$ becomes an element of a pencil an hence the union $X$ of all these $X_{\Sigma}$ is a closed subset of dimension $n-2$ in $H$. 

Observe that the quadrics that are smooth in $P$ form a non-empty open subset of the pencil $\langle g_1, g_2\rangle$, hence a non-empty open subset of a projective line. As a consequence, there is only a finite number of quadrics $\bar \Sigma$ in the pencil $\langle g_1, g_2\rangle$ that are singular in $P$. For every quadric $\bar \Sigma$ that is singular in $P$, the lines through $P$ that $\bar\Sigma$ contains form a proper closed subset $X_{\bar \Sigma}$ of $H$ of dimension $n-2$. Then, the finite union $Y$ of these closed subsets $X_{\bar\Sigma}$ with $X$ is a closed subset of $H$ of dimension $\leq n-2$. 
In conclusion, there is an open subset of lines through $P$ that are not contained in any quadric of $\langle g_1,g_2\rangle$.
\end{proof}

\begin{proposition}\label{prop: transversally}
Let $C\subset \mathbb P^n_K$ be a non-degenerate irreducible curve on at least two quadrics $g_1$ and $g_2$. There exists a line $L$ that is not contained in any quadric of the pencil $\langle g_1,g_2\rangle$ and intersects transversally $C$ only in a point $P$.
\end{proposition}

\begin{proof} 
As we observed in the previous theorem, lines through the selected point $P$ in $\mathbb P^n_K$ form a $\mathbb P^{n-1}_K$. Since we work in characteristic zero and since the general curve $C_{d-1}$ is smooth, all tangent lines are limits of secants. Hence the secant (or tangent) lines to $C_{d-1}$ passing through $P$ are in one-to-one correspondence with the points of $\pi(C_{d-1})$, the projection of $C_{d-1}$ from $P$  in $\mathbb P^{n-1}_K$. Any point in the Zariski-open set 
			$$U:= \mathbb P^{n-1}_K\backslash \pi(C_{d-1})\subset  \mathbb P^{n-1}_K $$
correspond to a line through $P$ which is neither secant nor tangent.
			
		Thanks to Theorem \ref{th:generic line}, there is a non empty Zarski-open set $V\subset  \mathbb P^{n-1}_K$ parametrizing lines through $P$ that are not contained in any quadric of the pencil $\langle g_1,g_2\rangle$.
		
		Any line in the open set $U\cap V \subset  \mathbb P^{n-1}_K$ is not contained in any quadric of the pencil $\langle g_1,g_2\rangle$ and intersects transversally $C_{d-1}$ only in the point $P$.
\end{proof}

\begin{theorem}\label{th:Interpolation}{\rm [Interpolation]}
The initial ideal of a general rational (resp.~elliptic) curve $C_d\subset\mathbb P^n_K$ is almost revlex for every $d$ such that $n\leq d < \frac{n^2+3n}{4}$ (resp.~$n < d < \frac{n^2+3n+2}{4}$).
\end{theorem}

\begin{proof}
We proceed by induction on $d$. If $d$ is the minimal possible degree, the statement holds thanks to Proposition \ref{prop: minimal degree}.

Now, assume that $d$ is not the minimal possible and let $C_{d-1}\subset\mathbb P^n_K$ be a general rational or elliptic curve of degree $d-1$ with initial ideal almost revlex, which exists by the induction hypothesis. 

By Corollary \ref{cor:grado iniziale} and by the hypotheses, we have $\alpha_d=\alpha_{d-1}=2$. Among the lowest $2d+1$ terms of degree $2$ there are at least two terms of degree $2$ that belong to the initial ideal of $C_{d-1}$. 

Let $\tau_1,\tau_{2}$ be the lowest terms of degree $2$ that belong to the initial ideal of $C_{d-1}$ and let $g_1,g_{2}$ be the two polynomials of the reduced Gr\"obner basis of $I(C_{d-1})$ with these terms as initial terms. The polynomials $g_1,g_2$ are linearly independent and generate a pencil $\langle g_1,g_2\rangle$ to which Theorem \ref{th:generic line} applies. Hence, there exists a line $L$ through a point $P$ of $C_{d-1}$ that is not contained in any quadric of the pencil.

Letting $s=\binom{n+2}{2}-p_{C_{d-1}}(2)$, the reduced Gr\"obner basis of $I(C_{d-1})$ is also made of other $s-2$ polynomials $g_3,\ldots,g_s$.

Let $Q_1$ be a point on $L$ such that $g_{(1)}(Q_1)\not=0$ and let $g_1^{(1)}:=g_1(Q_1)^{-1}g_1$ so that $g_1^{(1)}(Q_1)=1$. Moreover let  $g_2^{(1)}:=g_2-g_2(Q_1)g_1^{(1)}$, so that $g_2^{(1)}(Q_1)=0$. Since $g_2^{(1)}$ belongs to the pencil $\langle g_1,g_2\rangle$, there exists another point $Q_2$ of $L$ such that $g_2^{(1)}(Q_2)\not=0$. Let $g_2^{(2)}:=g_2^{(1)}(Q_2)^{-1}g_2^{(1)}$ so that $g_2^{(2)}(Q_2)=1$. 

For every $i\in \{3,\dots,s\}$, let $g_i^{(1)}:=g_i-g_i(Q_1)g_1^{(1)}$ and then $g_i^{(2)} := g_i^{(1)}-g_i^{(1)}(Q_2)g_2^{(2)}$, so that $g_i^{(2)}(Q_1)=g_i^{(2)}(Q_2)=0$. Hence, every polynomial $g_i^{(2)}$ vanishes on $C_{d-1}\cup L$ and $\mathrm{in}(g_i)=\mathrm{in}(g_i^{(2)})$ by construction. We can now conclude that the polynomials $g_3^{(2)},\dots,g_s^{(2)}$ are the polynomials of degree $2$ of the reduced Gr\"obner basis of the defining ideal of $C_{d-1}\cup L$ and their initial terms form a degrevlex segment with the same Hilbert function, at degree $t=2$, of a general rational (respectively, elliptic) curve $C_d$ of degree $d$.

Indeed, since moreover we can assume that the line $L$ intersects $C_{d-1}$ transversally thanks to Proposition \ref{prop: transversally}, we can now apply a result due to Sernesi \cite{S} and Hartshorne-Hirschowitz~\cite{HH} (see also \cite[Lemma 0.1]{BE1987}), which guarantees that the curve $C_{d-1}\cup L$ belongs to the irreducible component $\mathcal Z^n_{p_{C_d}(t)}$ of the Hilbert scheme containing the general rational (respectively, elliptic) curves of degree $d$.
 
Thus, $\mathcal Z^n_{p_{C_d}(t)}$ contains also the initial ideal of $C_{d-1}\cup L$, which is a strongly stable ideal that is formed by a degrevlex segment at degree $2$ by construction. So the double-generic initial ideal $\mathbf{G}$ of $\mathcal Z^n_{p_{C_d}(t)}$ must contain the same terms at degree $2$, due to Hilbert function reasonings and properties of the degrevlex term order (see \cite[Proposition 6]{BCR2017}), which are again based on Remark \ref{rem:drl}. 

The double-generic initial ideal $\mathbf G$ is the initial ideal of the general rational (respectively, elliptic) curve and we conclude obtaining the thesis thanks to Corollary \ref{cor: cor3} and Lemma~\ref{lemma:dgii}(ii).  
\end{proof}

\begin{example}
Starting from the rational normal curve $C_4\subset\mathbb P^4_K$, Theorem \ref{th:Interpolation} shows that the general rational curve $C_5$ has the almost revlex ideal as initial ideal removing the terms $\tau_1=x_2^2$, $\tau_2=x_2x_3$ from $\mathrm{in}(C_4)$ and the polynomials corresponding to the terms $\tau_1$ and $\tau_2$ from $I(C_4)$. Analogously, the same result holds for $C_6$. Although $\alpha_7=3>2$, we can state that even the initial ideal of $C_7$ is almost revlex thanks to Corollary \ref{cor: cor4}. However, for the case $C_7$ the construction of Theorem \ref{th:Interpolation} could also be performed starting from~$C_6$.
\end{example}

\begin{corollary}\label{cor:SLP}
The defining ideal $I(C_d)$ of a general rational (respectively elliptic) curve $C_d\subseteq \mathbb P^n_K$ of degree $d$  satisfies the $n$-SLP, for every $d$ such that $n\leq d < \frac{n^2+3n}{4}$ (respectively~$n < d < \frac{n^2+3n+2}{4}$).
\end{corollary}

\begin{proof}
It is enough to recall that $I(C_d)$ satisfies the $n$-SLP if and only if $\mathrm{gin}(C_d)=\mathrm{in}(C_d)$ satisfies the $n$-SLP (see \cite[Theorem 3.6]{PaTo}). Nevertheless, thanks to \cite[Theorem 4.6]{PaTo} $\mathrm{in}(C_d)$ satisfies the $n$-SLP because is almost revlex by Theorem \ref{th:Interpolation}.
\end{proof}

\begin{example}
Going back to Example \ref{ex: n=5 d=8}, we now know that the initial ideal $J$ of the general rational curve of degree $8$ in $\mathbb P^5$ is almost revlex. Then the ideal $J+(x_0,\dots,x_5)^4$ is almost revlex too (see \cite[Theorem 5.6]{PaTo}). Moreover it is Artinian with Hilbert function $H=(1,6,17,25)$. We note that $\Delta H=(1,5,11,8)$ is not quasi-symmetric, so that Theorem 25 of \cite{HW} is not an equivalence. 
\end{example}


\providecommand{\bysame}{\leavevmode\hbox to3em{\hrulefill}\thinspace}
\providecommand{\MR}{\relax\ifhmode\unskip\space\fi MR }
\providecommand{\MRhref}[2]{%
  \href{http://www.ams.org/mathscinet-getitem?mr=#1}{#2}
}
\providecommand{\href}[2]{#2}

\end{document}